\documentclass[11pt,a4paper]{article}
\usepackage[utf8]{inputenc}
\usepackage[T1]{fontenc}
\usepackage{lmodern}
\usepackage{amsmath,amsfonts,amssymb,amsthm}
\usepackage{graphicx,float}
\usepackage[hyphens]{url}
\usepackage{enumitem}
\usepackage{pifont}
\usepackage[pdfauthor={ANDRIAMANALINA Johariniaina Carlo},
            pdftitle={Classification of Finite-Dimensional Lie Algebras with Respect to the Length of Their Chief Series},
            pdfstartview=Fit,pdfpagelayout=SinglePage,pdfnewwindow=true,
            bookmarksnumbered=true,breaklinks,colorlinks,
            linkcolor=blue,urlcolor=black,citecolor=blue,linktoc=all]{hyperref}
\usepackage{geometry}
\theoremstyle{plain}
\newtheorem{theorem}{Theorem}[section]
\newtheorem{corollary}[theorem]{Corollary}
\newtheorem{proposition}[theorem]{Proposition}
\newtheorem{lemma}[theorem]{Lemma}

\theoremstyle{definition}
\newtheorem{definition}[theorem]{Definition}
\newtheorem{remark}[theorem]{Remark}

\setlist[itemize,1]{label={--}, itemsep=\baselineskip}
\setlist[enumerate,1]{label=\arabic*), itemsep=\baselineskip}

\title{Classification of Finite-Dimensional Lie Algebras with Respect to the Length of Their Chief Series}
\author{Johariniaina Carlo ANDRIAMANALINA}
\date{\today}

\begin{document}

\maketitle

\begin{abstract}
In this paper, we study finite-dimensional Lie algebras with respect to the length of their chief series. We examine separately the semisimple, solvable, and mixed cases over fields of characteristic \(0\), and we establish several structural results describing Lie algebras of prescribed chief length. In the solvable and mixed settings, the problem is reduced to the study of irreducible modules arising from suitable quotient actions. In positive characteristic, we investigate semisimple Lie algebras by means of Block's structural theorem and obtain several results concerning Lie algebras of chief length \(2\).
\end{abstract}

\section{Introduction}

The study of chains of substructures has long played an important role in algebra. In the theory of finite groups, several numerical invariants are defined in terms of chains of subgroups. Among the most prominent examples are the \emph{length} and \emph{depth} of a finite group, which measure the maximal and minimal lengths of unrefinable subgroup chains, respectively. These invariants have been investigated extensively; see, for example, \cite{burn} and the references therein. Results of this type provide valuable information on the global structure of a group and often lead to classifications of groups possessing small values of the corresponding invariant.

For Lie algebras, chains of ideals play a role analogous to chains of normal subgroups in group theory. In particular, a chief series is a finite chain of ideals whose successive factors are minimal ideals of the corresponding quotient algebras. Such series provide a natural way to decompose a Lie algebra into successive minimal ideal extensions. A theorem of Towers, recalled later in this paper, shows that any two chief series of a finite-dimensional Lie algebra have the same length. This common value, called the \emph{chief length} of the Lie algebra, is therefore a well-defined invariant that measures the complexity of its ideal structure. From this perspective, chief length may be viewed as a Lie-theoretic analogue of the chain-length invariants that arise in the study of finite groups.

The study of chief series and related chain invariants has already received some attention in the literature. In particular, Towers investigated the relationship between chief length and several other chain-length invariants arising from the subalgebra lattice of a finite-dimensional Lie algebra, including the minimum length of a maximal chain of subalgebras, chains of modular subalgebras, and chains of quasi-ideals \cite{chlg}. Among other results, he established characterisations of solvable Lie algebras in terms of these invariants and highlighted several striking differences between the Lie-theoretic and group-theoretic settings.

Motivated by this line of research, the present paper focuses specifically on chief length and investigates the extent to which finite-dimensional Lie algebras can be described according to the length of their chief series. Rather than seeking a complete classification in all situations, we aim to determine the structural constraints imposed by a prescribed chief length and to identify the extent to which the problem can be reduced to more fundamental questions in representation theory.

After introducing notation and recalling elementary facts concerning chief factors and chief series, we establish several preliminary results. In particular, we prove that a Lie algebra has chief length \(2\) if and only if every proper ideal is maximal. This criterion plays a central role throughout the paper.

We then study finite-dimensional semisimple Lie algebras over fields of characteristic \(0\). In contrast to the group-theoretic situation considered in \cite{burn}, where the classification of groups of a given length is highly intricate, the semisimple Lie algebra case admits a particularly simple description: we prove that a semisimple Lie algebra has chief length \(n\) if and only if it is the direct sum of exactly \(n\) simple ideals.

The next section is devoted to solvable Lie algebras over arbitrary fields. We first characterise solvable Lie algebras of chief length \(2\) and then obtain an inductive description of solvable Lie algebras of arbitrary chief length. More precisely, we show that a solvable Lie algebra has chief length \(n\) if and only if it arises as an extension of a Lie algebra of chief length \(n-1\) by an irreducible abelian ideal. This reduces the classification problem to the study of irreducible modules. Over algebraically closed fields of characteristic \(0\), Lie's theorem yields a stronger conclusion: the chief length of a solvable Lie algebra coincides with its dimension.

We subsequently investigate Lie algebras that are neither semisimple nor solvable. For chief length \(2\), we prove that such Lie algebras are precisely those of the form
\(
\mathcal{L}=S\ltimes\mathcal{R},
\)
where \(S\) is simple, \(\mathcal{R}\) is the abelian radical of \(\mathcal{L}\), and \(\mathcal{R}\) is an irreducible \(S\)-module. We then establish an inductive result analogous to the solvable case, showing that the general classification problem again reduces to the study of irreducible modules. Unlike the corresponding group-theoretic classifications of small depth and length obtained in \cite{burn}, this reduction appears to be the strongest general result that can be obtained.

The paper then turns to semisimple Lie algebras over fields of positive characteristic. Using Block's structural theorem, we derive several restrictions on semisimple Lie algebras of prescribed chief length. In particular, we obtain structural descriptions of semisimple Lie algebras of chief length \(2\) and show that further progress is closely connected with the study of derivation algebras and outer derivations of simple modular Lie algebras. We also prove that a semisimple Lie algebra has chief length equal to the number of simple components appearing in Block's decomposition if and only if it is a direct sum of simple Lie algebras.

The author was motivated in part by the treatments of ideals and structural properties of Lie algebras appearing in \cite{lati}, \cite{finnbid}, \cite{latchai}, and \cite{latnil}.

\section{Notation}

Throughout this paper, the symbol \(\oplus\) denotes the direct sum of algebras, \(\dotplus\) denotes the direct sum of vector spaces, and \(\ltimes\) denotes the semidirect product. Moreover, \(\{I_i\}\) denotes a sequence of ideals of a given algebra. All Lie algebras considered in this paper are finite-dimensional.

\section{Elementary Results}

\begin{definition}[\cite{lgchv}, p.2]
Let \(\mathcal{L}\) be a Lie algebra, and let \(B \subset A\) be subalgebras of \(\mathcal{L}\). The quotient \(A/B\) is called a \emph{chief factor} of \(\mathcal{L}\) if \(B\) is an ideal of \(\mathcal{L}\) and \(A/B\) is a minimal ideal of \(\mathcal{L}/B\).
\end{definition}

\begin{definition}
Let \(\mathcal{L}\) be a Lie algebra, and let
\(
\{0\}=I_0 \subsetneq I_1 \subsetneq \cdots \subsetneq I_n=\mathcal{L}
\)
be a chain of ideals of \(\mathcal{L}\) such that, for every \(i\), the quotient \(I_{i+1}/I_i\) is a chief factor of \(\mathcal{L}\). Such a chain is called a \emph{chief series} of \(\mathcal{L}\).
\end{definition}

As discussed in the introduction, chief length may be viewed as a Lie-theoretic analogue of the chain-length invariants arising in finite group theory. Such invariants measure the complexity of a group's structure through chains of subgroups; see \cite{burn}. In the Lie algebra setting, the corresponding role is played by chief series, which decompose a Lie algebra into successive minimal ideal extensions. We now make this notion precise.

\begin{theorem}[\cite{lgchv}, Theorem~3.1]
Let \(\mathcal{L}\) be a Lie algebra. If
\(
\{0\}=I_0 \subsetneq I_1 \subsetneq \cdots \subsetneq I_n=\mathcal{L}
\)
and
\(
\{0\}=J_0 \subsetneq J_1 \subsetneq \cdots \subsetneq J_m=\mathcal{L}
\)
are two chief series of \(\mathcal{L}\), then there exists a bijection between their chief factors such that corresponding factors are isomorphic as \(\mathcal{L}\)-modules.
\end{theorem}

\begin{corollary}
All chief series of a Lie algebra \(\mathcal{L}\) have the same length.
\end{corollary}

\begin{proof}
The result follows immediately from the previous theorem.
\end{proof}

\begin{definition}
The common length of the chief series of a Lie algebra \(\mathcal{L}\) is called the \emph{chief length} of \(\mathcal{L}\) and is denoted by \(l(\mathcal{L})\).
\end{definition}

By definition, a Lie algebra of chief length \(0\) is trivial, whereas a Lie algebra of chief length \(1\) is either simple or one-dimensional. The following theorem characterises Lie algebras of chief length \(2\).

\begin{theorem}\label{MAX}
A Lie algebra \(\mathcal{L}\) has chief length \(2\) if and only if every proper ideal of \(\mathcal{L}\) is maximal.
\end{theorem}

\begin{proof}
Suppose that \(l(\mathcal{L})=2\), and let \(I\) be a proper ideal of \(\mathcal{L}\). Consider the chain
\(
\{0\}=I_0 \subset I_1=I \subset I_2=\mathcal{L}.
\)
Since \(\mathcal{L}\) has chief length \(2\), this chain is a chief series of \(\mathcal{L}\). Hence, \(I_2/I_1\) is a chief factor, so \(\mathcal{L}/I\) is a minimal ideal of \(\mathcal{L}/I\). Therefore, no proper ideal of \(\mathcal{L}\) properly contains \(I\), and thus \(I\) is maximal.

Conversely, suppose that every proper ideal of \(\mathcal{L}\) is maximal. Let \(I\) be a proper ideal of \(\mathcal{L}\). Since \(I\) is maximal, the quotient \(\mathcal{L}/I\) is a minimal ideal of itself, and therefore \(I_2/I_1\) is a chief factor.

Now let \(J\) be an ideal such that
\(
\{0\} \subset J \subset I.
\)
By hypothesis, both \(I\) and \(J\) are maximal ideals. Since \(J \subset I\), it follows that \(J=I\). Hence, \(I\) contains no nontrivial proper ideals and is therefore a minimal ideal of \(\mathcal{L}\). Consequently, \(I_1/I_0\) is also a chief factor.

Thus,
\(
\{0\}=I_0 \subset I_1=I \subset I_2=\mathcal{L}
\)
is a chief series of \(\mathcal{L}\), and therefore \(l(\mathcal{L})=2\).
\end{proof}

\begin{remark}\label{nonoverlapid}
Let \(\mathcal{L}\) be a Lie algebra over an arbitrary field, and let \(I\) and \(J\) be ideals of \(\mathcal{L}\) such that
\(
I \cap J = \{0\}.
\)
Then
\(
[I,J]=\{0\},
\)
and consequently
\(
I \dotplus J = I \oplus J.
\)

Indeed, since \(I\) and \(J\) are ideals of \(\mathcal{L}\), we have
\(
[I,J] \subseteq I \cap J.
\)
Hence,
\(
[I,J]=\{0\}.
\)

Now let \(i_0+j_0,i_1+j_1 \in I \dotplus J\), where \(i_0,i_1 \in I\) and \(j_0,j_1 \in J\). Then
\(
[(i_0+j_0),(i_1+j_1)]
=
[i_0,i_1]+[j_0,j_1]+[i_0,j_1]+[j_0,i_1].
\)
Since \([I,J]=\{0\}\), it follows that
\(
[(i_0+j_0),(i_1+j_1)]
=
[i_0,i_1]+[j_0,j_1].
\)
Therefore, the Lie bracket on \(I \dotplus J\) coincides with the direct product bracket, and hence
\(
I \dotplus J = I \oplus J.
\)
\end{remark}

\section{Finite-Dimensional Semisimple Lie Algebras over Fields of Characteristic Zero}

We have already observed that a Lie algebra has chief length \(0\) if and only if it is trivial, whereas a Lie algebra has chief length \(1\) if and only if it is either one-dimensional or simple. In particular, a semisimple Lie algebra of chief length \(1\) is necessarily simple. In this section, we classify finite-dimensional semisimple Lie algebras over fields of characteristic zero according to their chief length.

\begin{proposition}
Let \(\mathcal{L}\) be a finite-dimensional semisimple Lie algebra over a field of characteristic zero. Then \(l(\mathcal{L})=n\) if and only if
\(
\mathcal{L}=I_1 \oplus I_2 \oplus \cdots \oplus I_n,
\)
where \(I_1,I_2,\ldots,I_n\) are simple ideals of \(\mathcal{L}\).
\end{proposition}

\begin{proof}
Since \(\mathcal{L}\) is a finite-dimensional semisimple Lie algebra over a field of characteristic zero, there exist simple ideals \(I_1,I_2,\ldots,I_m\) such that
\(
\mathcal{L}=I_1 \oplus I_2 \oplus \cdots \oplus I_m.
\)

Consider the chain of ideals
\(
\{0\}=J_0
\subsetneq
J_1=I_1
\subsetneq
J_2=I_1\oplus I_2
\subsetneq
\cdots
\subsetneq
J_m=\mathcal{L}.
\)

For each \(k \in \{0,1,\ldots,m-1\}\), we have
\(
J_{k+1}/J_k \cong I_{k+1},
\)
and since \(I_{k+1}\) is simple, the quotient \(J_{k+1}/J_k\) is a minimal ideal of \(\mathcal{L}/J_k\). Hence, each factor \(J_{k+1}/J_k\) is a chief factor, and therefore the above chain is a chief series of \(\mathcal{L}\).

Consequently,
\(
l(\mathcal{L})=m.
\)

It follows that \(\mathcal{L}\) has chief length \(n\) if and only if it can be expressed as a direct sum of exactly \(n\) simple ideals.
\end{proof}

\section{Solvable Lie algebras over arbitrary fields}
\label{secsolvretr}
Recall that the unique solvable Lie algebra of chief length \(0\) is the trivial Lie algebra, whereas every solvable Lie algebra of chief length \(1\) is one-dimensional.

\begin{lemma}
Let \(\mathcal{L}\) be a solvable Lie algebra of chief length \(2\). Then
\(
\dim \mathcal{L} \geq 2.
\)
\label{sovbvol}
\end{lemma}

\begin{proof}
A one-dimensional Lie algebra admits no proper non-trivial ideals. Hence, it cannot possess a chief series of length \(2\). Therefore, every solvable Lie algebra of chief length \(2\) has dimension at least \(2\).
\end{proof}

Recall that the derived series of a Lie algebra \(\mathcal{L}\), defined over a field of arbitrary characteristic, is given recursively by
\(
\mathcal{L}^{(0)}=\mathcal{L}, \qquad 
\mathcal{L}^{(1)}=[\mathcal{L},\mathcal{L}],
\)
and, for every \(i\in\mathbb{N}\),
\(
\mathcal{L}^{(i+1)}
=
[\mathcal{L}^{(i)},\mathcal{L}^{(i)}].
\)
If \(\mathcal{L}\) is solvable, the least integer \(n\in\mathbb{N}\) satisfying
\(
\mathcal{L}^{(n)}=\{0\}
\)
is called the \emph{solvable length} (or \emph{derived length}) of \(\mathcal{L}\).

\begin{proposition}\label{chieflengthtwo}
Let \(\mathcal{L}\) be a Lie algebra of chief length \(2\). Then the following assertions hold.

\begin{enumerate}
    \item If \(\mathcal{L}\) is solvable of solvable length \(n\), then
    \(
    1\leq n\leq 2.
    \)

    \item If \(\mathcal{L}\) is abelian, then
    \(
    \dim(\mathcal{L})=2.
    \)
\end{enumerate}
\end{proposition}

\begin{proof}
\begin{enumerate}
    \item Suppose that
    \(
    \mathcal{L}^{(2)}\neq \{0\}.
    \)
    Since \(\mathcal{L}\) is solvable,
    \(
    \mathcal{L}^{(1)}=[\mathcal{L},\mathcal{L}]
    \)
    is a proper ideal of \(\mathcal{L}\). Moreover,
    \(
    \mathcal{L}^{(2)}
    =
    [\mathcal{L}^{(1)},\mathcal{L}^{(1)}]
    \)
    is a proper ideal of \(\mathcal{L}^{(1)}\). Consequently,
    \(
    \mathcal{L}^{(2)}
    \subsetneq
    \mathcal{L}^{(1)}
    \subsetneq
    \mathcal{L}.
    \)
    This contradicts Theorem~\ref{MAX}, since \(\mathcal{L}\) has chief length \(2\). Hence,
    \(
    \mathcal{L}^{(2)}=\{0\},
    \)
    and therefore the solvable length of \(\mathcal{L}\) satisfies
    \(
    1\leq n\leq 2.
    \)

    \item Assume that \(\mathcal{L}\) is abelian and that
    \(
    \dim(\mathcal{L})>2.
    \)
    Then \(\mathcal{L}\) admits a two-dimensional proper vector subspace \(V\). Since \(\mathcal{L}\) is abelian, every vector subspace is an ideal; hence \(V\) is a proper ideal of \(\mathcal{L}\).

    Furthermore, \(V\) contains a one-dimensional subspace \(W\), which is likewise an ideal of \(\mathcal{L}\). Thus,
    \(
    W
    \subsetneq
    V
    \subsetneq
    \mathcal{L},
    \)
    contradicting Theorem~\ref{MAX}. Therefore,
    \(
    \dim(\mathcal{L})\leq 2.
    \)
    On the other hand, Lemma~\ref{sovbvol} yields
    \(
    \dim(\mathcal{L})\geq 2.
    \)
    Consequently,
    \(
    \dim(\mathcal{L})=2.
    \)
\end{enumerate}
\end{proof}

\begin{proposition}
Every two-dimensional Lie algebra is of chief length \(2\).
\end{proposition}

\begin{proof}
Let \(\mathcal{L}\) be a two-dimensional Lie algebra. Every proper ideal of \(\mathcal{L}\) is necessarily one-dimensional and therefore maximal. Hence, by Theorem~\ref{MAX}, \(\mathcal{L}\) has chief length \(2\).
\end{proof}

\begin{lemma}\label{decsolv}
Let \(\mathcal{L}\) be a solvable Lie algebra, defined over a field of arbitrary characteristic, and suppose that \(\mathcal{L}\) has chief length \(2\). If
\(
\mathcal{L}^{(1)}=[\mathcal{L},\mathcal{L}]\neq \{0\},
\)
then there exists
\(
0\neq a\in \mathcal{L}\setminus \mathcal{L}^{(1)}
\)
such that
\(
\mathcal{L}
=
\langle a\rangle \ltimes \mathcal{L}^{(1)},
\)
where \(\langle a\rangle\) denotes the one-dimensional Lie subalgebra generated by \(a\).
\end{lemma}

\begin{proof}
Since \(\mathcal{L}\) is solvable,
\(
\mathcal{L}^{(1)}=[\mathcal{L},\mathcal{L}]
\)
is a proper ideal of \(\mathcal{L}\). Hence, there exists a vector subspace \(V\) of \(\mathcal{L}\) such that
\(
\mathcal{L}
=
V+\mathcal{L}^{(1)}
\)
and
\(
V\cap \mathcal{L}^{(1)}=\{0\}.
\)

We claim that \(V\) is one-dimensional. Suppose, to the contrary, that
\(
\dim(V)\geq 2.
\)
Then there exist nonzero subspaces \(V_{1}\) and \(V_{2}\) satisfying
\(
V
=
V_{1}\dotplus V_{2}.
\)
Consequently,
\(
\mathcal{L}^{(1)}
\subsetneq
V_{2}+\mathcal{L}^{(1)}
\subsetneq
\mathcal{L}.
\)

Let \(x\in \mathcal{L}\) and \(v\in V_{2}+\mathcal{L}^{(1)}\). Since
\(
[x,v]\in [\mathcal{L},\mathcal{L}]
=
\mathcal{L}^{(1)},
\)
it follows that
\(
[x,v]\in V_{2}+\mathcal{L}^{(1)}.
\)
Therefore,
\(
V_{2}+\mathcal{L}^{(1)}
\)
is an ideal of \(\mathcal{L}\).

Hence,
\(
\mathcal{L}^{(1)}
\subsetneq
V_{2}+\mathcal{L}^{(1)}
\subsetneq
\mathcal{L},
\)
which contradicts Theorem~\ref{MAX}, since \(\mathcal{L}\) has chief length \(2\). Therefore,
\(
\dim(V)=1.
\)

Thus, there exists
\(
0\neq a\in \mathcal{L}\setminus \mathcal{L}^{(1)}
\)
such that
\(
V=\langle a\rangle.
\)
Since \(\mathcal{L}^{(1)}\) is an ideal of \(\mathcal{L}\), it follows that
\(
\mathcal{L}
=
\langle a\rangle \ltimes \mathcal{L}^{(1)}.
\)
\end{proof}

\begin{theorem} \label{vao2}
Let \(\mathcal{L}\) be a solvable Lie algebra over a field of arbitrary characteristic such that
\(
\mathcal{L}^{(1)}=[\mathcal{L},\mathcal{L}] \neq \{0\}.
\)
Then the following statements are equivalent:
\begin{enumerate}
    \item \(\mathcal{L}\) is of chief length \(2\).
    
    \item There exist a one-dimensional Lie algebra \(\langle a \rangle\) and an abelian Lie algebra \(\mathcal{L}^{(1)}\) such that
    \(
    \mathcal{L}=\langle a \rangle \ltimes \mathcal{L}^{(1)},
    \)
    where the action of \(\mathrm{ad}_a\) on \(\mathcal{L}^{(1)}\) is irreducible.
\end{enumerate}
\end{theorem}

\begin{proof}
\((1)\Rightarrow(2)\)

Assume that \(\mathcal{L}\) is of chief length \(2\). By Lemma \ref{decsolv},
\(
\mathcal{L}=\langle a \rangle \ltimes \mathcal{L}^{(1)}.
\)
Moreover, Proposition \ref{chieflengthtwo} implies that \(\mathcal{L}^{(1)}\) is abelian. It therefore remains to prove that the action of \(\mathrm{ad}_a\) on \(\mathcal{L}^{(1)}\) is irreducible.

Suppose, for contradiction, that the action is reducible. Then there exists a non-trivial proper subspace
\(
V\subsetneq \mathcal{L}^{(1)}
\)
such that
\(
[a,V]\subseteq V.
\)
Since \(\mathcal{L}^{(1)}\) is abelian,
\(
[\mathcal{L}^{(1)},V]\subseteq V.
\)
Consequently,
\(
[\mathcal{L},V]\subseteq V,
\)
and therefore \(V\) is an ideal of \(\mathcal{L}\). Since \(\mathcal{L}\) has chief length \(2\), every proper non-trivial ideal of \(\mathcal{L}\) is maximal by Theorem \ref{MAX}. This contradicts the existence of the proper non-trivial subspace
\(
V\subsetneq \mathcal{L}^{(1)}.
\)
Hence, the action of \(\mathrm{ad}_a\) on \(\mathcal{L}^{(1)}\) is irreducible.

\medskip

\((2)\Rightarrow(1)\)

Assume that
\(
\mathcal{L}=\langle a \rangle \ltimes \mathcal{L}^{(1)},
\)
where \(\mathcal{L}^{(1)}\) is abelian and the action of \(\mathrm{ad}_a\) on \(\mathcal{L}^{(1)}\) is irreducible.

Since 
\(
[\mathcal{L},\mathcal{L}]=\mathcal{L}^{(1)},
\)
 \(\mathcal{L}\) is solvable.

Since the action of \(\mathrm{ad}_a\) is irreducible and non-trivial, there exists
\(
x\in \mathcal{L}^{(1)}
\)
such that
\(
[a,x]=\mathrm{ad}_a(x)\neq 0.
\)
Thus, \(\mathcal{L}\) is non-abelian.

Let \(I\) be a proper non-trivial ideal of \(\mathcal{L}\). Then
\(
I\cap \mathcal{L}^{(1)}
\)
is an \(\mathrm{ad}_a\)-stable subspace of \(\mathcal{L}^{(1)}\). By irreducibility,
\(
I\cap \mathcal{L}^{(1)}=\{0\}
\quad \text{or} \quad
I\cap \mathcal{L}^{(1)}=\mathcal{L}^{(1)}.
\)

Assume first that
\(
I\cap \mathcal{L}^{(1)}=\{0\}.
\)
Since
\(
\mathcal{L}=\langle a \rangle \ltimes \mathcal{L}^{(1)},
\)
it follows that
\(
\mathcal{L}=I\dotplus \mathcal{L}^{(1)}.
\)
Because both \(I\) and \(\mathcal{L}^{(1)}\) are ideals, Remark \ref{nonoverlapid} implies that
\(
[I,\mathcal{L}^{(1)}]=\{0\}.
\)
Since \(I\) is one-dimensional and \(\mathcal{L}^{(1)}\) is abelian, it follows that \(\mathcal{L}\) is abelian, a contradiction.

Therefore,
\(
I\cap \mathcal{L}^{(1)}=\mathcal{L}^{(1)},
\)
and hence
\(
\mathcal{L}^{(1)}\subseteq I.
\)
Since \(I\) is proper and
\(
\mathcal{L}=\langle a \rangle \ltimes \mathcal{L}^{(1)},
\)
we obtain
\(
I=\mathcal{L}^{(1)}.
\)
Thus, \(\mathcal{L}^{(1)}\) is the unique proper non-trivial ideal of \(\mathcal{L}\). Consequently, every proper non-trivial ideal of \(\mathcal{L}\) is maximal. By Theorem \ref{MAX}, \(\mathcal{L}\) is of chief length \(2\).
\end{proof}

It is worth noting (but, perhaps, is not very surprising) that the classification of solvable Lie algebras of chief length 2 "almost" coincides with the classification of nonabelian solvable Lie algebras all whose proper subalgebras are abelian. See \cite{russia} and \cite{ernest} where a slightly more general classification of solvable Lie algebras all whose proper subalgebras are nilpotent is obtained. Namely, we have the following:

\begin{corollary}
A solvable Lie algebra of chief length \(2\) is a minimal non-abelian solvable Lie algebra. Furthermore, every minimal non-abelian solvable Lie algebra is either of chief length \(2\) or isomorphic to the three-dimensional Heisenberg algebra.
\end{corollary}

Consequently, a solvable Lie algebra is of chief length \(0\) if and only if it is trivial, and it is of chief length \(1\) if and only if it is one-dimensional. Moreover, a non-abelian solvable Lie algebra \(\mathcal{L}\) is of chief length \(2\) if and only if
\(
\mathcal{L}=\langle a \rangle \ltimes \mathcal{L}^{(1)},
\)
where \(\langle a\rangle\) is one-dimensional, \(\mathcal{L}^{(1)}\) is abelian, and \(\mathcal{L}^{(1)}\) is an irreducible \(\langle a\rangle\)-module under the action induced by \(\mathrm{ad}_a\).

These observations lead to the following.

Let \(A\) be an abelian ideal of a Lie algebra \(\mathcal{L}\). The quotient Lie algebra \(\mathcal{L}/A\) acts on \(A\) by
\(
(x+A)\cdot a=[x,a]
\)
for all \(x\in \mathcal{L}\) and \(a\in A\).

This action is well defined. Indeed, if
\(
x+A=y+A,
\)
then \(x-y\in A\). Since \(A\) is abelian, we have
\(
[x-y,a]=0
\)
for every \(a\in A\). Hence
\(
[x,a]=[y,a],
\)
so the action does not depend on the choice of representative.

\begin{theorem}
Let \(\mathcal{L}\) be a solvable Lie algebra over a field of arbitrary characteristic. Then \(\mathcal{L}\) has chief length \(n\) if and only if
\(
\mathcal{L}=\mathcal{H}\dotplus A,
\)
where \(A\) is an abelian ideal of \(\mathcal{L}\), \(\mathcal{H}\) is a vector space complement of \(A\) in \(\mathcal{L}\), \(A\) is an irreducible \(\mathcal{L}/A\)-module under the above action, and
\(
\mathcal{L}/A
\)
has chief length \(n-1\).
\label{ind1}
\end{theorem}

\begin{proof}
\((\Rightarrow)\)
Assume that \(\mathcal{L}\) has chief length \(n\). Since \(\mathcal{L}\) is solvable, there exists \(m\in \mathbb{N}\) such that
\(
\mathcal{L}^{(m)}
\)
is abelian. Let
\(
\{0\}=I_0\subsetneq I_1\subsetneq \cdots \subsetneq I_n=\mathcal{L}
\)
be a chief series of \(\mathcal{L}\) passing through
\(
\mathcal{L}^{(m)},
\)
that is, there exists \(p\) such that
\(
I_p=\mathcal{L}^{(m)}.
\)
Then \(I_1\) is abelian. Set
\(
A=I_1.
\)
It follows immediately that
\(
\mathcal{L}/A
\)
has chief length \(n-1\).

Suppose that \(A\) is reducible as an \(\mathcal{L}/A\)-module. Then there exists a non-trivial proper submodule
\(
M\subsetneq A.
\)
For every \(x\in \mathcal{L}\),
\(
(x+A)\cdot M=[x,M]\subseteq M.
\)
Hence, \(M\) is an ideal of \(\mathcal{L}\) strictly contained in \(A\), contradicting the fact that
\(
A/\{0\}=I_1/I_0
\)
is a chief factor. Therefore, \(A\) is an irreducible \(\mathcal{L}/A\)-module.

Finally, since \(A\) is an ideal of \(\mathcal{L}\), there exists a vector space complement \(\mathcal{H}\) of \(A\) in \(\mathcal{L}\) such that
\(
\mathcal{L}=\mathcal{H}\dotplus A.
\)

\medskip

\((\Leftarrow)\)
Assume that
\(
\mathcal{L}=\mathcal{H}\dotplus A,
\)
where \(A\) is an abelian ideal of \(\mathcal{L}\),
\(
A
\)
is an irreducible \(\mathcal{L}/A\)-module, and
\(
\mathcal{L}/A
\)
has chief length \(n-1\).

Let
\(
\{0\}=I_0\subsetneq I_1\subsetneq \cdots \subsetneq I_{n-1}=\mathcal{L}/A
\)
be a chief series of
\(
\mathcal{L}/A,
\)
and let
\(
\pi:\mathcal{L}\rightarrow \mathcal{L}/A
\)
denote the canonical projection. Define
\(
J_0=\{0\}, \qquad J_1=A,
\)
and, for \(1\leq i\leq n-1\),
\(
J_{i+1}=\pi^{-1}(I_i).
\)
Then
\(
\{0\}=J_0\subsetneq J_1\subsetneq \cdots \subsetneq J_n=\mathcal{L}.
\)

Since \(A\) is irreducible as an \(\mathcal{L}/A\)-module, it contains no non-trivial proper ideal of \(\mathcal{L}\). Hence,
\(
A/\{0\}
\)
is a chief factor.

By the correspondence theorem, the ideals of \(\mathcal{L}\) containing \(A\) are in bijective correspondence with the ideals of
\(
\mathcal{L}/A.
\)
Therefore, each factor
\(
J_{i+1}/J_i
\)
is chief. Consequently,
\(
\{0\}=J_0\subsetneq J_1\subsetneq \cdots \subsetneq J_n=\mathcal{L}
\)
is a chief series of \(\mathcal{L}\). Thus, \(\mathcal{L}\) has chief length \(n\).
\end{proof}

The preceding theorem shows that the study of solvable Lie algebras with prescribed chief length reduces to the study of irreducible modules of solvable Lie algebras. In general, no complete classification of such modules is known. However, over an algebraically closed field of characteristic \(0\), the following result holds.

Let \(\mathcal{L}\) be a solvable Lie algebra over an algebraically closed field \(\mathbb{F}\) of characteristic \(0\), and let \(M\) be an \(\mathcal{L}\)-module. Denote by
\(
\mathcal{L}^{*}=\operatorname{Hom}_{\mathbb{F}}(\mathcal{L},\mathbb{F})
\)
the dual space of \(\mathcal{L}\). For each \(\lambda\in \mathcal{L}^{*}\), define
\(
M_{\lambda}=\left\{m\in M \mid l\cdot m=\lambda(l)m \text{ for all } l\in \mathcal{L}\right\}.
\)

\begin{theorem}
Let \(\mathcal{L}\) be a solvable Lie algebra over an algebraically closed field of characteristic \(0\). Then \(\mathcal{L}\) has chief length \(n\) if and only if
\(
\dim(\mathcal{L})=n.
\)
\end{theorem}

\begin{proof}
We proceed by induction on \(n\).

For \(n=0\), the unique Lie algebra of chief length \(0\) is the trivial Lie algebra, whose dimension is \(0\). For \(n=1\), the solvable Lie algebras of chief length \(1\) are precisely the one-dimensional Lie algebras.

Assume that, for some \(n\in \mathbb{N}\), every solvable Lie algebra of chief length \(n\) has dimension \(n\), and let \(\mathcal{L}\) be a solvable Lie algebra of chief length \(n+1\). By Theorem \ref{ind1}, there exist an abelian ideal \(A\) of \(\mathcal{L}\) and a subalgebra \(\mathcal{H}\) such that
\(
\mathcal{L}=\mathcal{H}\ltimes A,
\)
where \(A\) is an irreducible \(\mathcal{L}/A\)-module and
\(
\mathcal{L}/A
\)
has chief length \(n\). By Lie theorem,
\(
\dim(A)=1.
\)
Moreover, the induction hypothesis yields
\(
\dim(\mathcal{L}/A)=n.
\)
Therefore,
\(
\dim(\mathcal{L})
=\dim(\mathcal{L}/A)+\dim(A)
=n+1.
\)

Conversely, assume that
\(
\dim(\mathcal{L})=n.
\)
Let
\(
\{0\}=I_0\subsetneq I_1\subsetneq \cdots \subsetneq I_r=\mathcal{L}
\)
be a chief series of \(\mathcal{L}\). For each \(1\leq i\leq r\), the factor
\(
I_i/I_{i-1}
\)
is a minimal ideal and therefore an irreducible \(\mathcal{L}/I_{i-1}\)-module. Since \(\mathcal{L}\) is solvable, each quotient
\(
\mathcal{L}/I_{i-1}
\)
is also solvable. Hence, by the previous corollary,
\(
\dim(I_i/I_{i-1})=1
\)
for every \(1\leq i\leq r\). Therefore,
\(
\dim(\mathcal{L})
=
\sum_{i=1}^{r}\dim(I_i/I_{i-1})
=
\sum_{i=1}^{r}1
=
r.
\)
Since \(\dim(\mathcal{L})=n\), we obtain \(r=n\). Thus, \(\mathcal{L}\) has chief length \(n\).
\end{proof}

\section{Neither semisimple nor solvable Lie algebras of characteristic zero}

Recall from Theorem \ref{MAX} that a Lie algebra has chief length two if and only if all of its ideals are maximal.

\begin{lemma}
Let \(\mathcal{L}\) be a Lie algebra of chief length two which is neither semisimple nor solvable. Then
\(
\mathcal{L}= S \ltimes \mathcal{R},
\)
where \(S\) is a simple subalgebra of \(\mathcal{L}\) and \(\mathcal{R}\) is the radical of \(\mathcal{L}\).
\label{dec}
\end{lemma}

\begin{proof}
By the Levi decomposition theorem,
\(
\mathcal{L}= \mathcal{L}_1 \ltimes \mathcal{R},
\)
where \(\mathcal{L}_1\) is a semisimple subalgebra of \(\mathcal{L}\). Since \(\mathcal{L}_1\) is semisimple, it decomposes as
\(
\mathcal{L}_1= S_1 \oplus S_2 \oplus \cdots \oplus S_n,
\)
where the \(S_i\)'s are simple ideals of \(\mathcal{L}_1\).

Assume that \(n>1\). Then
\(
S_1 \ltimes \mathcal{R}
\)
is a proper ideal of \(\mathcal{L}\). Since \(\mathcal{L}\) is not semisimple, its radical \(\mathcal{R}\) is a non-zero proper ideal of \(\mathcal{L}\). Consequently,
\(
\mathcal{R} \subsetneq S_1 \ltimes \mathcal{R} \subsetneq \mathcal{L},
\)
which yields a chain of distinct proper ideals of \(\mathcal{L}\). This contradicts the assumption that \(\mathcal{L}\) has chief length two. Therefore, \(n=1\), and hence \(\mathcal{L}_1\) is simple. The result follows.
\end{proof}

\begin{lemma}
Let \(\mathcal{R}\) be the radical of a Lie algebra \(\mathcal{L}\) of chief length two which is neither semisimple nor solvable. Then \(\mathcal{R}\) is abelian.
\label{abelrad}
\end{lemma}

\begin{proof}
Since \(\mathcal{L}\) is neither semisimple nor solvable, \(\mathcal{R}\) is a non-zero proper ideal of \(\mathcal{L}\). Suppose that \(\mathcal{R}\) is non-abelian. Then its derived algebra \(\mathcal{R}^{(1)}\) is non-zero. Since \(\mathcal{R}\) is solvable,
\(
\mathcal{R}^{(1)} \subsetneq \mathcal{R}.
\)
Moreover, \(\mathcal{R}^{(1)}\) is an ideal of \(\mathcal{L}\). Hence,
\(
\mathcal{R}^{(1)} \subsetneq \mathcal{R} \subsetneq \mathcal{L},
\)
which shows that \(\mathcal{R}^{(1)}\) is a proper ideal of \(\mathcal{L}\) that is not maximal. This contradicts the assumption that \(\mathcal{L}\) has chief length two.
\end{proof}

\begin{lemma}
Let
\(
\mathcal{L}=S \ltimes \mathcal{R},
\)
where \(S\) is a simple Lie subalgebra and \(\mathcal{R}\) is the radical of \(\mathcal{L}\). Let \(I\) be a proper ideal of \(\mathcal{L}\). Then either
\(
I \subseteq \mathcal{R}
\)
or
\(
\mathcal{L}= I + \mathcal{R}.
\)
\label{Jradnasimp}
\end{lemma}

\begin{proof}
Let \(I\) be a proper ideal of \(\mathcal{L}\). Then \(I + \mathcal{R}\) is an ideal of \(\mathcal{L}\) containing \(\mathcal{R}\). Consequently,
\(
(I + \mathcal{R})/\mathcal{R}
\)
is an ideal of \(\mathcal{L}/\mathcal{R}\).

Since
\(
\mathcal{L}=S \ltimes \mathcal{R},
\)
we have
\(
\mathcal{L}/\mathcal{R} \cong S.
\)
As \(S\) is simple, \(\mathcal{L}/\mathcal{R}\) is also simple. Therefore,
\(
(I + \mathcal{R})/\mathcal{R}
\)
is either trivial or equal to \(\mathcal{L}/\mathcal{R}\). Hence,
\(
I + \mathcal{R}= \mathcal{R}
\)
or
\(
I + \mathcal{R}= \mathcal{L}.
\)
Equivalently,
\(
I \subseteq \mathcal{R}
\)
or
\(
\mathcal{L}= I + \mathcal{R}.
\)
\end{proof}

\begin{lemma}
Let
\(
\mathcal{L}=S \ltimes \mathcal{R},
\)
where \(S\) is a simple Lie subalgebra and \(\mathcal{R}\) is the abelian radical of \(\mathcal{L}\). Suppose that the action of \(S\) on \(\mathcal{R}\) is irreducible. Then either \(\mathcal{R}\) is one-dimensional or \(\mathcal{R}\) is the unique proper ideal of \(\mathcal{L}\).
\label{farny}
\end{lemma}

\begin{proof}
Let \(I\) be a proper ideal of \(\mathcal{L}\). Then \(I \cap \mathcal{R}\) is an ideal contained in \(\mathcal{R}\). Since both \(I \cap \mathcal{R}\) and \(\mathcal{R}\) are ideals of \(\mathcal{L}\), they are naturally \(S\)-modules. Hence, \(I \cap \mathcal{R}\) is an \(S\)-submodule of \(\mathcal{R}\).

By hypothesis, \(\mathcal{R}\) is an irreducible \(S\)-module. Therefore,
\(
I \cap \mathcal{R}= \{0\}
\)
or
\(
I \cap \mathcal{R}= \mathcal{R}.
\)

We consider these two cases separately.

\medskip

\noindent
\textbf{Case 1.} Suppose that
\(
I \cap \mathcal{R}= \{0\}.
\)

Since \(I\) is non-zero and proper, Lemma \ref{Jradnasimp} implies that
\(
I + \mathcal{R}= \mathcal{L}.
\)
Therefore,
\(
\mathcal{L}= I \ltimes \mathcal{R}.
\)

On the other hand,
\(
\mathcal{L}=S \ltimes \mathcal{R}.
\)
Hence, by the Levi–Malcev theorem, there exists an automorphism
\(
f:\mathcal{L}\to \mathcal{L}
\)
such that
\(
f(S)=I
\quad \text{and} \quad
f(\mathcal{R})=\mathcal{R}.
\)

Since both \(I\) and \(\mathcal{R}\) are ideals of \(\mathcal{L}\), Remark \ref{nonoverlapid} yields
\(
[I,\mathcal{R}]=\{0\}.
\)
Consequently,
\begin{align*}
f([S,\mathcal{R}])
&=[f(S),f(\mathcal{R})] \\
&=[I,\mathcal{R}] \\
&=\{0\}.
\end{align*}

As \(f\) is an automorphism, it follows that
\(
[S,\mathcal{R}]=\{0\}.
\)
Thus, \(\mathcal{R}\) is a trivial \(S\)-module. Since \(\mathcal{R}\) is also irreducible, it follows that \(\mathcal{R}\) is one-dimensional.

\medskip

\noindent
\textbf{Case 2.} Suppose that
\(
I \cap \mathcal{R}= \mathcal{R}.
\)

Then
\(
\mathcal{R}\subseteq I.
\)
By Lemma \ref{Jradnasimp}, either
\(
I \subseteq \mathcal{R}
\)
or
\(
I+\mathcal{R}= \mathcal{L}.
\)

Since \(\mathcal{R}\subseteq I\), the latter equality implies that
\(
I=\mathcal{L},
\)
contradicting the assumption that \(I\) is proper. Therefore,
\(
I \subseteq \mathcal{R}.
\)
Hence,
\(
I=\mathcal{R}.
\)
\end{proof}

\begin{theorem}
Let \(\mathcal{L}\) be a Lie algebra which is neither semisimple nor solvable. Then \(\mathcal{L}\) has chief length two if and only if
\(
\mathcal{L}=S \ltimes \mathcal{R},
\)
where \(S\) is a simple subalgebra, \(\mathcal{R}\) is the abelian radical of \(\mathcal{L}\), and the action of \(S\) on \(\mathcal{R}\) is irreducible.
\label{mixbak}
\end{theorem}

\begin{proof}
Suppose first that \(\mathcal{L}\) has chief length two. Since \(\mathcal{L}\) is neither semisimple nor solvable, Lemma \ref{dec} yields
\(
\mathcal{L}=S \ltimes \mathcal{R},
\)
where \(S\) is a simple subalgebra and \(\mathcal{R}\) is the radical of \(\mathcal{L}\). Moreover, Lemma \ref{abelrad} shows that \(\mathcal{R}\) is abelian.

Since \(\mathcal{R}\) is an ideal of \(\mathcal{L}\), it is naturally an \(S\)-module. Suppose that \(\mathcal{R}\) admits a proper non-zero \(S\)-submodule \(\mathcal{R}'\). Then, for all \(x\in S\) and \(y\in \mathcal{R}'\),
\(
[x,y]\in \mathcal{R}'.
\)

Since \(\mathcal{R}\) is abelian, for all \(x'\in \mathcal{R}\) and \(y\in \mathcal{R}'\),
\(
[x',y]=0\in \mathcal{R}'.
\)

As
\(
\mathcal{L}=S \ltimes \mathcal{R},
\)
it follows that, for all \(a\in \mathcal{L}\) and \(y\in \mathcal{R}'\),
\(
[a,y]\in \mathcal{R}'.
\)
Hence, \(\mathcal{R}'\) is an ideal of \(\mathcal{L}\). Since
\(
\mathcal{R}'\subsetneq \mathcal{R}\subsetneq \mathcal{L},
\)
this contradicts the assumption that \(\mathcal{L}\) has chief length two. Therefore, \(\mathcal{R}\) is an irreducible \(S\)-module.

Conversely, let \(S\) be a simple Lie algebra and let \(\mathcal{R}\) be an irreducible \(S\)-module. Endow \(\mathcal{R}\) with the trivial Lie bracket, and consider
\(
\mathcal{L}=S \ltimes \mathcal{R},
\)
where the bracket between an element of \(S\) and an element of \(\mathcal{R}\) is given by the action of \(S\) on \(\mathcal{R}\).

By Lemma \ref{farny}, either \(\mathcal{R}\) is one-dimensional or \(\mathcal{R}\) is the unique proper ideal of \(\mathcal{L}\). In the latter case, \(\mathcal{L}\) clearly has chief length two. It therefore remains to consider the case where \(\mathcal{R}\) is one-dimensional.

Assume that \(\dim(\mathcal{R})=1\), and let \(I\) be a proper ideal of
\(
\mathcal{L}=S \ltimes \mathcal{R}.
\)
Then \(I\cap S\) is an ideal of \(S\). Since \(S\) is simple,
\(
I\cap S=S
\)
or
\(
I\cap S=\{0\}.
\)

We examine these two cases separately.

\medskip

\noindent
\textbf{Case 1.} Suppose that
\(
I\cap S=\{0\}.
\)

Since
\(
\mathcal{L}=S \ltimes \mathcal{R}
\)
and \(\dim(\mathcal{R})=1\),
\(
\dim(\mathcal{L})=\dim(S)+1.
\)
Moreover, \(I\neq \{0\}\) and
\(
I\cap S=\{0\}.
\)
Hence,
\(
I+S=\mathcal{L},
\)
which implies that \(I\) is one-dimensional. Consequently, \(I\) is solvable. Since \(\mathcal{R}\) is the radical of \(\mathcal{L}\),
\(
I\subseteq \mathcal{R}.
\)
As \(\mathcal{R}\) is one-dimensional, it follows that
\(
I=\mathcal{R}.
\)

\medskip

\noindent
\textbf{Case 2.} Suppose that
\(
I\cap S=S.
\)

Then
\(
S\subseteq I.
\)
Since
\(
\dim(\mathcal{L})=\dim(S)+1,
\)
the inclusion
\(
S\subsetneq I
\)
would imply that
\(
I=\mathcal{L},
\)
contradicting the assumption that \(I\) is proper. Therefore,
\(
I=S.
\)

Consequently, the only proper ideals of \(\mathcal{L}\) are \(S\) and \(\mathcal{R}\). Since
\(
S\cap \mathcal{R}=\{0\},
\)
both are maximal ideals of \(\mathcal{L}\). Hence, \(\mathcal{L}\) has chief length two.
\end{proof}

To summarise, a Lie algebra of chief length zero is trivial, and a Lie algebra of chief length one is either simple or one-dimensional. Moreover, a Lie algebra of chief length two which is neither semisimple nor solvable is precisely of the form
\(
\mathcal{L}=S \ltimes \mathcal{R},
\)
where \(S\) is a simple subalgebra, \(\mathcal{R}\) is the abelian radical of \(\mathcal{L}\), and the action of \(S\) on \(\mathcal{R}\) is irreducible. These results lead to the following theorem.

Recall that if \(\mathcal{A}\) is an abelian ideal of a Lie algebra \(\mathcal{L}\), then the quotient algebra \(\mathcal{L}/\mathcal{A}\) acts naturally on \(\mathcal{A}\) via
\(
(l+\mathcal{A})a=[l,a]
\)
for all \(l\in \mathcal{L}\) and \(a\in \mathcal{A}\). As observed previously, this action is well defined.

\begin{theorem}
Let \(\mathcal{L}\) be a Lie algebra which is neither semisimple nor solvable. Then \(\mathcal{L}\) has chief length \(n \in \mathbb{N}\) if and only if
\(
\mathcal{L}= \mathcal{H} \dotplus \mathcal{A},
\)
where \(\mathcal{A}\) is a non-zero abelian ideal, $\mathcal{H}$ is a vector-space complement of $\mathcal{A}$ in $\mathcal{L}$,
\(
l(\mathcal{L}/\mathcal{A})=n-1,
\)
\(\mathcal{L}/\mathcal{A}\) is non-solvable, and \(\mathcal{A}\) is an irreducible \(\mathcal{L}/\mathcal{A}\)-module.
\label{ind2}
\end{theorem}

\begin{proof}
Since \(\mathcal{L}\) is not semisimple, its radical is non-zero. As the radical is solvable, it contains a non-zero abelian ideal, which is therefore an ideal of \(\mathcal{L}\). Let \(\mathcal{A}\) be a minimal non-zero abelian ideal of \(\mathcal{L}\).

Applying the same argument as in Theorem \ref{ind1}, the result follows.
\end{proof}

Again, this theorem shows that the classification of mixed Lie algebras with respect to chief length reduces to the classification of irreducible modules. However, no general classification of irreducible modules exists, and such a classification is not attainable in full generality. Consequently, this theorem is essentially the strongest general result obtainable for mixed Lie algebras over fields of characteristic \(0\).

It is also worth noting that the previous theorem remains valid in positive characteristic. Hence, the structural constraints governing the study of mixed Lie algebras are identical in both characteristic zero and positive characteristic.

\section{Semisimple Lie algebras of positive characteristic}

Section \ref{secsolvretr} studied solvable Lie algebras over fields of arbitrary characteristic. Moreover, the paragraph preceding this section states that theorem \ref{ind2} also holds in characteristic \(p > 0\). Consequently, theorem \ref{ind2} together with the subsequent discussion provide the constraints for the classification of Lie algebras that are neither semisimple nor solvable with respect to their chief length in positive characteristic. Therefore, it remains to study semisimple Lie algebras of characteristic \(p > 0\) from the perspective of their chief length.\\

Recall that, by theorem \ref{MAX}, a Lie algebra has chief length \(2\) if and only if all its ideals are maximal.\\

Let \(\mathbb{F}\) be a field of characteristic \(p>0\). Throughout this section, \(S\) and \(S_i\) denote simple Lie algebras over \(\mathbb{F}\), and \(B_n\) denotes the quotient algebra
\(
B_n=\mathbb{F}[X_1,\ldots,X_n]/(X_1^p,X_2^p,\ldots,X_n^p),
\)
where \(n\in\mathbb{N}\), \(\mathbb{F}[X_1,\ldots,X_n]\) is the polynomial algebra in the indeterminates \(X_i\), and \((X_1^p,X_2^p,\ldots,X_n^p)\) is the ideal generated by the elements \(X_i^p\). When \(n=0\), we set \(B_0=\mathbb{F}\).

\begin{theorem}[\cite{block}, Theorem 7.1]
Let \(\mathcal{S}\) be a simple Lie algebra of characteristic \(p > 0\). Then
\(
\mathrm{Der}(\mathcal{S}\otimes B_n)
=
\mathrm{Der}(\mathcal{S})\otimes B_n
+
\Gamma\otimes \mathrm{Der}(B_n),
\)
where \(\Gamma\) denotes the centroid of \(\mathcal{S}\).
\label{decSBn}
\end{theorem}

\begin{definition}
Let \(\mathcal{L}\) and \(\mathcal{L}'\) be Lie algebras such that \(\mathcal{L}\subseteq \mathrm{Der}(\mathcal{L}')\). An ideal \(I\) of \(\mathcal{L}'\) is called an \(\mathcal{L}\)-ideal of \(\mathcal{L}'\) if it is invariant under the action of \(\mathcal{L}\), that is,
\(
D(I)\subseteq I
\qquad
\text{for every } D\in\mathcal{L}.
\)
The Lie algebra \(\mathcal{L}'\) is said to be \(\mathcal{L}\)-simple if its only \(\mathcal{L}\)-ideals are
\(
0
\qquad\text{and}\qquad
\mathcal{L}'.
\)
\end{definition}

\begin{theorem}[Block, {\cite[Theorem 9.3]{block}}]
Let \(\mathcal{L}\) be a semisimple Lie algebra over a field of characteristic \(p>0\). Then there exist \(n\in\mathbb{N}\), simple Lie algebras \(\mathcal{S}_i\), and non-negative integers \(n_i\) such that
\(
\bigoplus_{i=1}^{n}\mathrm{inDer}(\mathcal{S}_i\otimes B_{n_i})
\subseteq
\mathcal{L}
\subseteq
\bigoplus_{i=1}^{n}\mathrm{Der}(\mathcal{S}_i\otimes B_{n_i}).
\)
Conversely, let \(\mathcal{L}\) be a Lie algebra satisfying the above inclusions. For each \(i\), let \(\mathcal{L}_i\) denote the projection of \(\mathcal{L}\) onto \(\mathrm{Der}(\mathcal{S}_i\otimes B_{n_i})\). Then \(\mathcal{L}\) is semisimple if and only if \(\mathcal{S}_i\otimes B_{n_i}\) is \(\mathcal{L}_i\)-simple for every \(i\).
\label{block1}
\end{theorem}

Assume the notation of theorem \ref{block1}. For each \(i\), define
\(
\mathcal{L}^i
=
\mathcal{L}\cap \mathrm{Der}(\mathcal{S}_i\otimes B_{n_i}).
\)

\begin{lemma}
Let \(\mathcal{L}\) be a semisimple Lie algebra over a field of characteristic \(p>0\). Then
\(
\mathcal{L}^i\neq 0
\qquad
\text{for all } i.
\)
\label{Li}
\end{lemma}

\begin{proof}
By theorem \ref{block1},
\(
\bigoplus_{i=1}^{n}\mathrm{inDer}(\mathcal{S}_i\otimes B_{n_i})
\subseteq
\mathcal{L}.
\)
Hence,
\(
\mathrm{inDer}(\mathcal{S}_i\otimes B_{n_i})
\subseteq
\mathcal{L}
\qquad
\text{for all } i.
\)
Therefore,
\(
\mathrm{inDer}(\mathcal{S}_i\otimes B_{n_i})
=
\mathrm{inDer}(\mathcal{S}_i\otimes B_{n_i})
\cap
\mathrm{Der}(\mathcal{S}_i\otimes B_{n_i})
\subseteq
\mathcal{L}\cap \mathrm{Der}(\mathcal{S}_i\otimes B_{n_i})
=
\mathcal{L}^i,
\)
which implies that \(\mathcal{L}^i\neq 0\).
\end{proof}

\begin{lemma}
Assume the notation of theorem \ref{block1} and lemma \ref{Li}. Let
\(
(i_k)_{1\leqslant k\leqslant r}
\subseteq
\{1,\ldots,n\},
\qquad
r\leqslant n.
\)
Then
\(
\bigoplus_{k=1}^{r}\mathcal{L}^{i_k}
\)
is an ideal of \(\mathcal{L}\).
\label{Liideal}
\end{lemma}

\begin{proof}
Since
\(
\bigoplus_{i=1}^{n}\mathrm{Der}(\mathcal{S}_i\otimes B_{n_i})
\)
is a direct sum of Lie algebras,
\(
\mathrm{Der}(\mathcal{S}_{i_k}\otimes B_{n_{i_k}})
\)
is an ideal of
\(
\bigoplus_{i=1}^{n}\mathrm{Der}(\mathcal{S}_i\otimes B_{n_i})
\)
for every \(i_k\). Consequently,
\(
\mathcal{L}^{i_k}
=
\mathcal{L}\cap
\mathrm{Der}(\mathcal{S}_{i_k}\otimes B_{n_{i_k}})
\)
is an ideal of \(\mathcal{L}\) for every \(k\). Since a finite sum of ideals is again an ideal, it follows that
\(
\bigoplus_{k=1}^{r}\mathcal{L}^{i_k}
\)
is an ideal of \(\mathcal{L}\).
\end{proof}

\begin{proposition}
Let \(\mathcal{L}\) be a semisimple Lie algebra over a field of characteristic \(p>0\). If \(\mathcal{L}\) has chief length \(2\), then either
\(
\mathcal{S}\otimes B_n
\cong
\mathrm{inDer}(\mathcal{S}\otimes B_n)
\subseteq
\mathcal{L}
\subseteq
\mathrm{Der}(\mathcal{S}\otimes B_n),
\)
or
\(
\mathcal{L}
=
\mathcal{S}_1\oplus \mathcal{S}_2,
\)
where \(\mathcal{S}\), \(\mathcal{S}_1\), and \(\mathcal{S}_2\) are simple Lie algebras.
\label{decMAX}
\end{proposition}

\begin{proof}
Let \(\mathcal{L}\) be a semisimple Lie algebra of characteristic \(p>0\) and chief length \(2\). Assume the notation of theorem \ref{block1} and lemma \ref{Liideal}. By lemma \ref{Liideal},
\(
\bigoplus_{k=1}^{r}\mathcal{L}^{i_k}
\)
is an ideal of \(\mathcal{L}\) for every choice of indices \((i_k)\).

Suppose that \(n\geqslant 3\). Then
\(
\mathcal{L}^1
\subsetneq
\mathcal{L}^1\oplus \mathcal{L}^2
\subsetneq
\mathcal{L},
\)
which yields a chain of proper ideals of \(\mathcal{L}\). This contradicts the assumption that \(\mathcal{L}\) has chief length \(2\). Hence,
\(
n\leqslant 2.
\)

Assume first that \(n=2\). The same argument shows that
\(
\mathcal{L}
=
\mathcal{L}^1\oplus \mathcal{L}^2.
\)
Moreover,
\(
\mathcal{S}_i\otimes B_{n_i}
\cong
\mathrm{inDer}(\mathcal{S}_i\otimes B_{n_i})
\subseteq
\mathcal{L}^i
\subseteq
\mathrm{Der}(\mathcal{S}_i\otimes B_{n_i})
\qquad
(i\in\{1,2\}).
\)

By theorem \ref{decSBn},
\(
\mathrm{Der}(\mathcal{S}_i\otimes B_{n_i})
=
\mathrm{Der}(\mathcal{S}_i)\otimes B_{n_i}
+
\Gamma\otimes \mathrm{Der}(B_{n_i}),
\)
where \(\Gamma\) denotes the centroid of \(\mathcal{S}_i\). In particular,
\(
\mathrm{Der}(\mathcal{S}_i)\otimes B_{n_i}
\)
is an ideal of
\(
\mathrm{Der}(\mathcal{S}_i\otimes B_{n_i}).
\)
Consequently,
\(
\bigl(\mathrm{Der}(\mathcal{S}_i)\otimes B_{n_i}\bigr)\cap \mathcal{L}^i
\)
is an ideal of \(\mathcal{L}^i\). Since
\(
\mathcal{L}
=
\mathcal{L}^1\oplus \mathcal{L}^2,
\)
every ideal of \(\mathcal{L}^i\) is also an ideal of \(\mathcal{L}\). Because \(\mathcal{L}\) has chief length \(2\), the algebra \(\mathcal{L}^i\) admits no non-trivial proper ideals. Therefore,
\(
\bigl(\mathrm{Der}(\mathcal{S}_i)\otimes B_{n_i}\bigr)\cap \mathcal{L}^i
=
\mathcal{L}^i.
\)
Hence,
\(
\mathcal{L}^i
\subseteq
\mathrm{Der}(\mathcal{S}_i)\otimes B_{n_i}.
\)

Since
\(
\mathrm{inDer}(\mathcal{S}_i\otimes B_{n_i})
\)
is an ideal of
\(
\mathrm{Der}(\mathcal{S}_i)\otimes B_{n_i},
\)
it follows that
\(
\mathrm{inDer}(\mathcal{S}_i\otimes B_{n_i})
\)
is an ideal of \(\mathcal{L}^i\). Again using the fact that \(\mathcal{L}^i\) has no non-trivial proper ideals, we obtain
\(
\mathcal{L}^i
=
\mathrm{inDer}(\mathcal{S}_i\otimes B_{n_i}).
\)

Suppose now that \(n_i\neq 0\). Then the associative algebra \(B_{n_i}\) possesses a non-trivial proper ideal \(P\). Consequently,
\(
\mathcal{S}_i\otimes P
\)
is a non-trivial proper ideal of
\(
\mathcal{S}_i\otimes B_{n_i}
\cong
\mathrm{inDer}(\mathcal{S}_i\otimes B_{n_i})
=
\mathcal{L}^i,
\)
which is impossible. Therefore,
\(
n_1=n_2=0.
\)

Hence,
\(
\mathcal{L}^i
\subseteq
\mathrm{Der}(\mathcal{S}_i\otimes \mathbb{F})
=
\mathrm{Der}(\mathcal{S}_i)\otimes \mathbb{F}
\cong
\mathrm{Der}(\mathcal{S}_i).
\)
Thus,
\(
\mathrm{inDer}(\mathcal{S}_i)
\subseteq
\mathcal{L}^i
\subseteq
\mathrm{Der}(\mathcal{S}_i),
\)
and
\(
\mathcal{L}
=
\mathcal{L}^1\oplus \mathcal{L}^2.
\)

Since
\(
\mathrm{inDer}(\mathcal{S}_i)
\)
is an ideal of
\(
\mathrm{Der}(\mathcal{S}_i),
\)
it is also an ideal of \(\mathcal{L}^i\), and therefore an ideal of \(\mathcal{L}\). If
\(
\mathrm{inDer}(\mathcal{S}_{i_0})
\subsetneq
\mathcal{L}^{i_0}
\)
for some \(i_0\in\{1,2\}\), then
\(
\mathrm{inDer}(\mathcal{S}_{i_0})
\subsetneq
\mathcal{L}^{i_0}
\subsetneq
\mathcal{L}
\)
is a chain of proper ideals, contradicting the fact that \(\mathcal{L}\) has chief length \(2\). Hence,
\(
\mathcal{L}^i
=
\mathrm{inDer}(\mathcal{S}_i)
\cong
\mathcal{S}_i,
\)
and therefore
\(
\mathcal{L}
=
\mathcal{S}_1\oplus \mathcal{S}_2.
\)

Finally, if \(n=1\), then \(\mathcal{L}\) is of the form
\(
\mathcal{S}\otimes B_n
\cong
\mathrm{inDer}(\mathcal{S}\otimes B_n)
\subseteq
\mathcal{L}
\subseteq
\mathrm{Der}(\mathcal{S}\otimes B_n),
\)
as required.
\end{proof}

\begin{theorem} \label{aoan}
Every Lie algebra of the form
\(
\mathcal{L}
=
\mathcal{S}_1\oplus \mathcal{S}_2,
\)
where \(\mathcal{S}_1\) and \(\mathcal{S}_2\) are simple Lie algebras, is semisimple and has chief length \(2\).
\end{theorem}

\begin{proof}
The proof is identical to that of the corresponding result for semisimple Lie algebras of characteristic \(0\).
\end{proof}

It remains to study Lie algebras satisfying
\(
\mathcal{S}\otimes B_n
\cong
\mathrm{inDer}(\mathcal{S}\otimes B_n)
\subseteq
\mathcal{L}
\subseteq
\mathrm{Der}(\mathcal{S}\otimes B_n).
\)

\begin{theorem}
Let \(\mathcal{L}\) be a semisimple Lie algebra satisfying
\(
\mathcal{S}\otimes B_n
\cong
\mathrm{inDer}(\mathcal{S}\otimes B_n)
\subseteq
\mathcal{L}
\subseteq
\mathrm{Der}(\mathcal{S}\otimes B_n).
\)
Then the following statements are equivalent:
\begin{enumerate}
    \item \(\mathcal{L}\) has chief length \(2\);
    
    \item the quotient Lie algebra
    \(
    \mathcal{L}/\mathrm{inDer}(\mathcal{S}\otimes B_n)
    \)
    is either simple or one-dimensional.
\end{enumerate}
\label{qosim=max}
\end{theorem}

\begin{proof}
Assume first that \(\mathcal{L}\) has chief length \(2\). Let
\(
J
\)
be a proper ideal of
\(
\mathcal{L}/\mathrm{inDer}(\mathcal{S}\otimes B_n),
\)
and let
\(
\pi^{-1}(J)
\)
denote its preimage in \(\mathcal{L}\). Then
\(
\mathrm{inDer}(\mathcal{S}\otimes B_n)
\subsetneq
\pi^{-1}(J)
\subsetneq
\mathcal{L},
\)
which yields a chain of proper ideals of \(\mathcal{L}\). This contradicts the assumption that \(\mathcal{L}\) has chief length \(2\). Hence,
\(
\mathcal{L}/\mathrm{inDer}(\mathcal{S}\otimes B_n)
\)
has no non-trivial proper ideals. Therefore, it is either simple or one-dimensional.

Conversely, assume that
\(
\mathcal{L}/\mathrm{inDer}(\mathcal{S}\otimes B_n)
\)
is either simple or one-dimensional. Since
\(
\mathcal{L}
\subseteq
\mathrm{Der}(\mathcal{S}\otimes B_n),
\)
theorem \ref{block1} implies that
\(
\mathrm{inDer}(\mathcal{S}\otimes B_n)
\cong
\mathcal{S}\otimes B_n
\)
is \(\mathcal{L}\)-simple. Consequently,
\(
\mathrm{inDer}(\mathcal{S}\otimes B_n)
\)
contains no non-trivial proper ideals.

Let \(I\) be a proper ideal of \(\mathcal{L}\) such that
\(
I
\nsubseteq
\mathrm{inDer}(\mathcal{S}\otimes B_n).
\)
Then
\(
\frac{I+\mathrm{inDer}(\mathcal{S}\otimes B_n)}
{\mathrm{inDer}(\mathcal{S}\otimes B_n)}
\)
is a non-zero ideal of
\(
\mathcal{L}/\mathrm{inDer}(\mathcal{S}\otimes B_n).
\)
Since the quotient algebra is either simple or one-dimensional, it follows that
\(
I+\mathrm{inDer}(\mathcal{S}\otimes B_n)
=
\mathcal{L}.
\)
Moreover,
\(
I\cap \mathrm{inDer}(\mathcal{S}\otimes B_n)
\)
is an ideal of
\(
\mathrm{inDer}(\mathcal{S}\otimes B_n).
\)
As
\(
\mathrm{inDer}(\mathcal{S}\otimes B_n)
\)
is \(\mathcal{L}\)-simple, we must have
\(
I\cap \mathrm{inDer}(\mathcal{S}\otimes B_n)=0.
\)
Therefore,
\(
\mathcal{L}
=
I\oplus \mathrm{inDer}(\mathcal{S}\otimes B_n).
\)

It follows that every proper ideal of \(\mathcal{L}\) is either
\(
\mathrm{inDer}(\mathcal{S}\otimes B_n)
\)
or a complementary ideal \(I\) satisfying
\(
\mathcal{L}
=
I\oplus \mathrm{inDer}(\mathcal{S}\otimes B_n).
\)
Hence, every non-trivial proper ideal of \(\mathcal{L}\) is maximal, and therefore \(\mathcal{L}\) has chief length \(2\).
\end{proof}

Theorem \ref{qosim=max} shows that the classification of semisimple Lie algebras of chief length \(2\) reduces to the study of the quotient algebra
\(
\mathcal{L}/\mathrm{inDer}(\mathcal{S}\otimes B_n).
\)
By theorem \ref{decSBn},
\[
\mathrm{Der}(\mathcal{S}\otimes B_n)
=
\bigl(\mathrm{Der}(\mathcal{S})\otimes B_n\bigr)
\dotplus
\bigl(\Gamma\otimes \mathrm{Der}(B_n)\bigr),
\]
where \(\Gamma\) denotes the centroid of \(\mathcal{S}\). Consequently,
\(
\mathcal{L}/\mathrm{inDer}(\mathcal{S}\otimes B_n)
\)
may be identified with a subalgebra of
\(
\bigl(\mathrm{outDer}(\mathcal{S})\otimes B_n\bigr)
\dotplus
\bigl(\Gamma\otimes \mathrm{Der}(B_n)\bigr).
\)

Therefore, the classification problem reduces to determining the simple or one-dimensional subalgebras of
\(
\bigl(\mathrm{outDer}(\mathcal{S})\otimes B_n\bigr)
\dotplus
\bigl(\Gamma\otimes \mathrm{Der}(B_n)\bigr)
\)
for which
\(
\mathrm{inDer}(\mathcal{S}\otimes B_n)
\)
is an irreducible module. In this paper, we restrict our attention to the case
\(
\mathcal{L}
=
\mathrm{Der}(\mathcal{S}\otimes B_n).
\)

Recall that \(B_n= \mathbb{F}[X_1, \cdots , X_n]/(X_1^p,\cdots,X_n^p)\) where \(p\) is the characteristic of \(\mathbb{F}\). Define the map
\(
\frac{d}{dx_i} : B_n \longrightarrow B_n
\)
by
\(
\frac{d}{dx_i}
\left(
\sum_{\alpha} a_{\alpha}
x_1^{\alpha_1}x_2^{\alpha_2}\cdots x_n^{\alpha_n}
\right)
=
\sum_{\alpha}
a_{\alpha}\alpha_i\,
x_1^{\alpha_1}\cdots
x_i^{\alpha_i-1}\cdots
x_n^{\alpha_n},
\)
where
\(
\alpha=(\alpha_1,\alpha_2,\dots,\alpha_n)\in \mathbb{N}^n
\)
and \(a_\alpha \in F\).

\begin{proposition}
\(
\mathrm{Der}(B_n)
=
\sum_{i=1}^{n} B_n\frac{d}{dX_i}.
\)
That is, \(\mathrm{Der}(B_n)\) is generated, as a \(B_n\)-module, by the derivations
\(
\frac{d}{dX_i}.
\)
Moreover, \(\mathrm{Der}(B_n)\) is a simple Lie algebra.
\label{DerBn}
\end{proposition}

\begin{proof}
The Lie algebra
\(
\mathrm{Der}(B_n)
\)
is the Jacobson--Witt algebra \(W(n,1)\); see \cite{slg}.
\end{proof}

\begin{lemma}
\(
\mathrm{inDer}(\mathcal{S}\otimes B_n)
=
\mathrm{inDer}(\mathcal{S})\otimes B_n.
\)
\label{inincin}
\end{lemma}

\begin{proof}
The Lie algebra
\(
\mathrm{inDer}(\mathcal{S}\otimes B_n)
\)
is generated by the inner derivations
\(
\mathrm{ad}_{s\otimes b},
\qquad
s\in \mathcal{S},
\quad
b\in B_n.
\)
Let
\(
s_1\otimes b_1\in \mathcal{S}\otimes B_n.
\)
Then
\(
\mathrm{ad}_{s\otimes b}(s_1\otimes b_1)
=
[s\otimes b,s_1\otimes b_1]
=
[s,s_1]\otimes bb_1.
\)
On the other hand,
\(
(\mathrm{ad}_s\otimes b)(s_1\otimes b_1)
=
\mathrm{ad}_s(s_1)\otimes bb_1
=
[s,s_1]\otimes bb_1.
\)
Hence,
\(
\mathrm{ad}_{s\otimes b}
=
\mathrm{ad}_s\otimes b.
\)
Therefore,
\(
\mathrm{inDer}(\mathcal{S}\otimes B_n)
\subseteq
\mathrm{inDer}(\mathcal{S})\otimes B_n.
\)

Conversely, every element of
\(
\mathrm{inDer}(\mathcal{S})\otimes B_n
\)
is a linear combination of maps of the form
\(
\mathrm{ad}_s\otimes b
=
\mathrm{ad}_{s\otimes b},
\)
which belong to
\(
\mathrm{inDer}(\mathcal{S}\otimes B_n).
\)
Thus,
\(
\mathrm{inDer}(\mathcal{S})\otimes B_n
\subseteq
\mathrm{inDer}(\mathcal{S}\otimes B_n).
\)

Consequently,
\(
\mathrm{inDer}(\mathcal{S}\otimes B_n)
=
\mathrm{inDer}(\mathcal{S})\otimes B_n.
\)
\end{proof}

\begin{theorem}
Let
\(
\mathcal{L}
=
\mathrm{Der}(\mathcal{S}\otimes B_n)
\)
be a Lie algebra over a field of characteristic \(p>0\). Then the following statements are equivalent:
\begin{enumerate}
    \item \(\mathcal{L}\) has chief length \(2\);
    
    \item
    \(
    \mathrm{Der}(\mathcal{S})
    =
    \mathrm{inDer}(\mathcal{S}).
    \)
\end{enumerate}
\end{theorem}

\begin{proof}
Assume first that
\(
\mathcal{L}
=
\mathrm{Der}(\mathcal{S}\otimes B_n)
\)
has chief length \(2\). By theorem \ref{decSBn},
\(
\mathrm{Der}(\mathcal{S}\otimes B_n)
=
\bigl(\mathrm{Der}(\mathcal{S})\otimes B_n\bigr)
+
\bigl(\Gamma\otimes \mathrm{Der}(B_n)\bigr),
\)
where \(\Gamma\) denotes the centroid of \(\mathcal{S}\). In particular,
\(
\mathrm{Der}(\mathcal{S})\otimes B_n
\)
is an ideal of \(\mathcal{L}\).

By lemma \ref{inincin},
\(
\mathrm{inDer}(\mathcal{S}\otimes B_n)
=
\mathrm{inDer}(\mathcal{S})\otimes B_n.
\)
Hence,
\(
\mathrm{inDer}(\mathcal{S}\otimes B_n)
\subseteq
\mathrm{Der}(\mathcal{S})\otimes B_n.
\)
Since
\(
\mathrm{inDer}(\mathcal{S}\otimes B_n)
\)
is a proper ideal of \(\mathcal{L}\), the quotient
\(
\frac{\mathrm{Der}(\mathcal{S})\otimes B_n}
{\mathrm{inDer}(\mathcal{S}\otimes B_n)}
\)
is an ideal of
\(
\frac{\mathrm{Der}(\mathcal{S}\otimes B_n)}
{\mathrm{inDer}(\mathcal{S}\otimes B_n)}.
\)

By theorem \ref{qosim=max},
\(
\frac{\mathrm{Der}(\mathcal{S}\otimes B_n)}
{\mathrm{inDer}(\mathcal{S}\otimes B_n)}
\)
is simple. Moreover,
\(
\mathrm{Der}(\mathcal{S})\otimes B_n
\subsetneq
\mathrm{Der}(\mathcal{S}\otimes B_n),
\)
and therefore
\(
\frac{\mathrm{Der}(\mathcal{S})\otimes B_n}
{\mathrm{inDer}(\mathcal{S}\otimes B_n)}
=
0.
\)
Consequently,
\(
\mathrm{Der}(\mathcal{S})\otimes B_n
=
\mathrm{inDer}(\mathcal{S}\otimes B_n)
=
\mathrm{inDer}(\mathcal{S})\otimes B_n,
\)
which implies that
\(
\mathrm{Der}(\mathcal{S})
=
\mathrm{inDer}(\mathcal{S}).
\)

Conversely, assume that
\(
\mathrm{Der}(\mathcal{S})
=
\mathrm{inDer}(\mathcal{S}).
\)
By theorem \ref{decSBn},
\(
\mathcal{L}
=
\mathrm{Der}(\mathcal{S}\otimes B_n)
=
\bigl(\mathrm{Der}(\mathcal{S})\otimes B_n\bigr)
+
\bigl(\Gamma\otimes \mathrm{Der}(B_n)\bigr).
\)
Using the hypothesis, we obtain
\(
\mathcal{L}
=
\bigl(\mathrm{inDer}(\mathcal{S})\otimes B_n\bigr)
+
\bigl(\Gamma\otimes \mathrm{Der}(B_n)\bigr).
\)

Moreover,
\(
\bigl(\mathrm{inDer}(\mathcal{S})\otimes B_n\bigr)
\cap
\bigl(\Gamma\otimes \mathrm{Der}(B_n)\bigr)
=
0.
\)
Hence,
\(
\mathcal{L}
=
\bigl(\mathrm{inDer}(\mathcal{S})\otimes B_n\bigr)
\dotplus
\bigl(\Gamma\otimes \mathrm{Der}(B_n)\bigr).
\)

By lemma \ref{inincin},
\(
\mathrm{inDer}(\mathcal{S}\otimes B_n)
=
\mathrm{inDer}(\mathcal{S})\otimes B_n.
\)
Therefore,
\(
\frac{\mathrm{Der}(\mathcal{S}\otimes B_n)}
{\mathrm{inDer}(\mathcal{S}\otimes B_n)}
\cong
\Gamma\otimes \mathrm{Der}(B_n).
\)

Since \(\Gamma\) is a field extension of \(\mathbb{F}\),
\(
\Gamma\otimes \mathrm{Der}(B_n)
\)
is simple whenever
\(
\mathrm{Der}(B_n)
\)
is simple. By proposition \ref{DerBn},
\(
\mathrm{Der}(B_n)
\)
is simple. Hence,
\(
\frac{\mathrm{Der}(\mathcal{S}\otimes B_n)}
{\mathrm{inDer}(\mathcal{S}\otimes B_n)}
\)
is simple. Theorem \ref{qosim=max} now implies that
\(
\mathcal{L}
=
\mathrm{Der}(\mathcal{S}\otimes B_n)
\)
has chief length \(2\).
\end{proof}

The previous theorem shows that the classification of semisimple Lie algebras of chief length \(2\) and of that type reduces to determining the simple Lie algebras of positive characteristic satisfying
\(
\mathrm{Der}(\mathcal{S})
=
\mathrm{inDer}(\mathcal{S}).
\)

We now generalise the preceding results. Throughout the following discussion, we retain the notation of theorem \ref{block1}.

\begin{lemma}
Let \(\mathcal{L}\) be a semisimple Lie algebra over a field of characteristic \(p>0\). Assume the notation of theorem \ref{block1}. Then
\(
n\leqslant l(\mathcal{L}).
\)
\end{lemma}

\begin{proof}
Suppose, for contradiction, that
\(
n>l(\mathcal{L}).
\)
By lemma \ref{Liideal},
\(
\bigoplus_{k=1}^{r}\mathcal{L}^{i_k}
\)
is an ideal of \(\mathcal{L}\) for every \(r\). In particular, consider the chain
\(
0
=
I_0
\subsetneq
I_1
=
\mathcal{L}^1
\subsetneq
I_2
=
\mathcal{L}^1\oplus \mathcal{L}^2
\subsetneq
\cdots
\subsetneq
I_n
=
\bigoplus_{i=1}^{n}\mathcal{L}^i.
\)
This is a strictly increasing chain of ideals of length \(n\). Since
\(
n>l(\mathcal{L}),
\)
this chain is longer than a chief series of \(\mathcal{L}\), which is impossible. Therefore,
\(
n\leqslant l(\mathcal{L}).
\)
\end{proof}

\begin{theorem}
Let \(\mathcal{L}\) be a Lie algebra over a field of characteristic \(p>0\). Assume the notation of theorem \ref{block1}, and let \(n\) denote the number of direct summands appearing in the decomposition of theorem \ref{block1}. Then the following statements are equivalent:
\begin{enumerate}
    \item \(\mathcal{L}\) is semisimple and
    \(
    l(\mathcal{L})=n;
    \)
    
    \item
    \(
    \mathcal{L}
    =
    \bigoplus_{i=1}^{n}\mathcal{S}_i,
    \)
    where each \(\mathcal{S}_i\) is a simple Lie algebra.
\end{enumerate}
\end{theorem}

\begin{proof}
Assume first that \(\mathcal{L}\) is semisimple and satisfies
\(
l(\mathcal{L})=n.
\)
The case \(n=1\) is immediate. Suppose therefore that
\(
n>1.
\)

Consider the chain of ideals
\(
0
=
I_0
\subsetneq
I_1
=
\mathcal{L}^1
\subsetneq
I_2
=
\mathcal{L}^1\oplus \mathcal{L}^2
\subsetneq
\cdots
\subsetneq
I_n
=
\bigoplus_{i=1}^{n}\mathcal{L}^i.
\)
If
\(
\bigoplus_{i=1}^{n}\mathcal{L}^i
\subsetneq
\mathcal{L},
\)
then that chain could be extended to a strictly longer chain of ideals in \(\mathcal{L}\), contradicting the assumption that
\(
l(\mathcal{L})=n.
\)
Hence,
\(
\mathcal{L}
=
\bigoplus_{i=1}^{n}\mathcal{L}^i.
\)

Applying the same argument used in the proof of proposition \ref{decMAX}, each algebra \(\mathcal{L}^i\) must coincide with a simple Lie algebra \(\mathcal{S}_i\). Therefore,
\(
\mathcal{L}
=
\bigoplus_{i=1}^{n}\mathcal{S}_i.
\)

Conversely, suppose that
\(
\mathcal{L}
=
\bigoplus_{i=1}^{n}\mathcal{S}_i,
\)
where each \(\mathcal{S}_i\) is simple. Then \(\mathcal{L}\) is semisimple. Moreover, the proof is identical to that of the corresponding theorem for semisimple Lie algebras of characteristic \(0\). Consequently,
\(
l(\mathcal{L})=n.
\)
\end{proof}

\begin{proposition}
Let \(\mathcal{L}\) be a semisimple Lie algebra over a field of characteristic \(p>0\), and assume the notation of theorem \ref{block1}. Suppose that
\(
n<l(\mathcal{L}).
\)
Let
\(
(i_k)_{1\leqslant k\leqslant r}
\subseteq
\{1,\ldots,n\}.
\)
Then
\(
\bigoplus_{k=1}^{r}\mathrm{inDer}(\mathcal{S}_{i_k}\otimes B_{n_{i_k}})
\)
is an ideal of \(\mathcal{L}\).
\end{proposition}

\begin{proof}
For every \(k\),
\(
\mathrm{inDer}(\mathcal{S}_{i_k}\otimes B_{n_{i_k}})
\)
is an ideal of
\(
\mathrm{Der}(\mathcal{S}_{i_k}\otimes B_{n_{i_k}}).
\)
Since
\(
\bigoplus_{i=1}^{n}\mathrm{Der}(\mathcal{S}_i\otimes B_{n_i})
\)
is a direct sum of Lie algebras,
\(
\mathrm{inDer}(\mathcal{S}_{i_k}\otimes B_{n_{i_k}})
\)
is also an ideal of
\(
\bigoplus_{i=1}^{n}\mathrm{Der}(\mathcal{S}_i\otimes B_{n_i}).
\)
Consequently,
\(
\bigoplus_{k=1}^{r}\mathrm{inDer}(\mathcal{S}_{i_k}\otimes B_{n_{i_k}})
\)
is an ideal of
\(
\bigoplus_{i=1}^{n}\mathrm{Der}(\mathcal{S}_i\otimes B_{n_i}).
\)

Since
\(
\mathcal{L}
\subseteq
\bigoplus_{i=1}^{n}\mathrm{Der}(\mathcal{S}_i\otimes B_{n_i}),
\)
the intersection
\(
\mathcal{L}
\cap
\bigoplus_{k=1}^{r}\mathrm{inDer}(\mathcal{S}_{i_k}\otimes B_{n_{i_k}})
\)
is an ideal of \(\mathcal{L}\). Moreover,
\(
\bigoplus_{k=1}^{r}\mathrm{inDer}(\mathcal{S}_{i_k}\otimes B_{n_{i_k}})
\subseteq
\mathcal{L},
\)
and therefore
\(
\mathcal{L}
\cap
\bigoplus_{k=1}^{r}\mathrm{inDer}(\mathcal{S}_{i_k}\otimes B_{n_{i_k}})
=
\bigoplus_{k=1}^{r}\mathrm{inDer}(\mathcal{S}_{i_k}\otimes B_{n_{i_k}}).
\)
Hence,
\(
\bigoplus_{k=1}^{r}\mathrm{inDer}(\mathcal{S}_{i_k}\otimes B_{n_{i_k}})
\)
is an ideal of \(\mathcal{L}\).
\end{proof}

When
\(
n<l(\mathcal{L}),
\)
the previous lemma yields the chain of ideals
\(
I_1
=
\mathrm{inDer}(\mathcal{S}_1\otimes B_{n_1})
\subsetneq
I_2
=
\mathrm{inDer}(\mathcal{S}_1\otimes B_{n_1})
\oplus
\mathrm{inDer}(\mathcal{S}_2\otimes B_{n_2})
\subsetneq
\cdots
\subsetneq
I_n
=
\bigoplus_{i=1}^{n}\mathrm{inDer}(\mathcal{S}_i\otimes B_{n_i}),
\)
whose length is at most \(n\). To refine this chain into a chief series, one may proceed as follows. First, determine a chief series of
\(
\mathcal{L}/I_n.
\)
The corresponding preimages then provide all ideals in a chief series of \(\mathcal{L}\) containing the chain
\(
\{I_i\}.
\)
The remaining step consists in determining whether additional ideals occur inside the family
\(
\{I_i\}.
\)
Consequently, the next stage of the analysis reduces to the study of
\(
\mathcal{L}/I_n.
\)

Observe that
\(
\frac{
\bigoplus_{i=1}^{n}
\mathrm{Der}(\mathcal{S}_i\otimes B_{n_i})
}{
I_n}
\cong
\bigoplus_{i=1}^{n}
\frac{\mathrm{Der}(\mathcal{S}_i\otimes B_{n_i})}
{\mathrm{inDer}(\mathcal{S}_i\otimes B_{n_i})}.
\)
By theorem \ref{decSBn},
\(
\mathrm{Der}(\mathcal{S}_i\otimes B_{n_i})
=
\bigl(
\mathrm{Der}(\mathcal{S}_i)\otimes B_{n_i}
\bigr)
+
\bigl(
\Gamma_i\otimes \mathrm{Der}(B_{n_i})
\bigr).
\)
Hence,
\(
\frac{\mathrm{Der}(\mathcal{S}_i\otimes B_{n_i})}
{\mathrm{inDer}(\mathcal{S}_i\otimes B_{n_i})}
\cong
\bigl(
\mathrm{outDer}(\mathcal{S}_i)\otimes B_{n_i}
\bigr)
+
\bigl(
\Gamma_i\otimes \mathrm{Der}(B_{n_i})
\bigr).
\)
Therefore,
\(
\mathcal{L}/I_n
\)
is a subalgebra of
\(
\bigoplus_{i=1}^{n}
\left(
\bigl(
\mathrm{outDer}(\mathcal{S}_i)\otimes B_{n_i}
\bigr)
+
\bigl(
\Gamma_i\otimes \mathrm{Der}(B_{n_i})
\bigr)
\right).
\)

\bibliographystyle{plain}
\bibliography{biblio}

\end{document}